\documentclass[12pt,reqno,intlim]{amsart}

\usepackage{amssymb,amsmath}

\usepackage[arrow,matrix,curve]{xy}

\newdir{ >}{{}*!/-5pt/\dir{>}}

\newcommand{\nc}{\newcommand}

\nc{\bC}{\bold{C}} \nc{\bN}{\Bbb{N}} \nc{\cF}{\mathcal{F}}
\nc{\cE}{\mathcal{E}} \nc{\cR}{\mathcal{R}} \nc{\cM}{\mathcal{M}}
\nc{\al}{\alpha} \nc{\bt}{\beta} \nc{\gm}{\gamma} \nc{\dl}{\delta}
\nc{\om}{\omega} \nc{\sg}{\sigma} \nc{\Sg}{\Sigma} \nc{\vf}{\varphi}
\nc{\ve}{\varepsilon} \nc{\os}{\overset} \nc{\ol}{\overline}
\nc{\ul}{\underline} \nc{\us}{\underset} \nc{\sbs}{\subset}
\nc{\bsl}{\backslash} \nc{\Ra}{\Rightarrow}
\nc{\lra}{\longrightarrow} \nc{\all}{\allowdisplaybreaks}

\nc{\Codes}{\operatorname{{\bold{Codes}}}}
\nc{\RegMono}{\operatorname{\mathcal{R}{\rm{eg}\mathcal{M}{\rm{ono}\!}}}}
\nc{\RegEpi}{\operatorname{\mathcal{R}{\rm{eg}\mathcal{E}{\rm{pi}\!}}}}
\nc{\Mn}{\operatorname{\mathcal{M}{\rm{ono}\!}}}
\nc{\Ep}{\operatorname{\mathcal{E}{\rm{pi}\!}}}
\nc{\Rg}{\operatorname{\mathcal{R}{\rm{eg}\!}}}
\nc{\Ob}{\operatorname{Ob\!}}

\numberwithin{equation}{section}

\newtheorem{theo}{\ \ \ Theorem}[section]
\newtheorem{lem}[theo]{\ \ \ Lemma}
\newtheorem{prop}[theo]{\ \ \ Proposition}
\newtheorem{cor}[theo]{\ \ \ Corollary}

\theoremstyle{definition}
\newtheorem{exmp}[theo]{\ \ \ Example}

\theoremstyle{remark}
\newtheorem{rem}[theo]{\ \ \ Remark}

\begin{document}

\title[]
{The criteria for the uniqueness of a weight homomorphism of a baric algebra}

\author{Dali Zangurashvili}

\maketitle

\begin{abstract}
The criteria for a baric algebra $A$ (over a field $K$) to have a unique weight homomorphism are found. One of them requires a certain system of equations to have a unique non-trivial solution in the field $K$. Applying this criterion, we provide an example showing that Holgate's well-known sufficient condition for the uniqueness of a weight homomorphism  is not necessary, and give also a new example of a baric algebra with two weight homomorphisms. Another criterion found in this paper asserts that a baric algebra has a unique weight homomorphism if and only if the transition matrix from any semi-natural basis $B_1$ to any semi-natural basis $B_2$ is stochastic.
\bigskip


\end{abstract}

\section{Introduction}

The notion of a baric algebra was introduced by Etherington in 1939 for algebraic genetics purposes \cite{E}. It was defined as a finite-dimensional algebra $A$ over a field $K$ such that there exists a non-trivial algebra homomorphism $A\rightarrow K$ (called a weight homomorphism). In the same paper, Etherington noted that a baric algebra may have more than one weight homomorphism by providing the corresponding sufficient condition for the case of commutative associative algebras. In 1969, Holgate gave a sufficient condition for a baric algebra to have a unique weight homomorphism \cite{H}; it requires the algebra to have a weight homomorphism with the nil kernel.

In the present paper, some necessary and sufficient conditions for the uniqueness of a weight homomorphism of a baric algebra are found. One of them requires the following system of equations to have a unique non-trivial solution in the field $K$:
\begin{equation}
x_ix_j=\sum_{k=1}^{n}\gamma_{ijk}x_k \; \; \; \; (1\leq i,j\leq n)
\end{equation}
\noindent where $n$ is the dimension of the algebra, while $\gamma_{ijk}$ are the structural constants with respect to a basis. Applying this criterion, we provide an example showing that the above-mentioned sufficient condition for the uniqueness of a weight homomorphism by Holgate is not necessary. Moreover, we give an example of a baric algebra with two weight homomorphisms that does not fall under the scope of Etherington's above-mentioned result. 

Another criterion found in this paper  asserts that \textit{a baric algebra has a unique weight homomorphism if and only if the transition matrix from any semi-natural basis $B_1$ to any semi-natural basis $B_2$ is stochastic}.

\vskip+2mm
The author gratefully acknowledges the financial support from Shota
Rustaveli National Science Foundation of Georgia (Ref.: STEM-22-
1601).

\section{Weight homomorphisms and the Etherington's system of equations}
Let $K$ be a field. If it is not stated otherwise, all  algebras and matrices considered in the paper are assumed to be respectively algebras and matrices over $K$. At that, algebras are not assumed to be associative or commutative.


  Let $A$ be a finite-dimentional algebra, $e_1,e_2,...,e_n$ be its basis, and let $\gamma_{ijk}$ ($1\leq i,j,k\leq n$) be the structural constants with respect to this basis. In \cite{E}, Etherington considered the following system of equations:
 \begin{equation}
x_ix_j=\sum_{i=1}^{n}\gamma_{ijk}x_k \; \; \; \; (1\leq i,j\leq n)
\end{equation}
\vskip+2mm
\textbf{Throughout the paper, under a solution of system (2.1) we mean its solution in the field} $\mathbf{K}$.
\vskip+2mm
Let $W$ be the set of all weight homomorphisms of the algebra $A$. The following proposition immediately follows from the arguments given in the paper \cite{E} by Etherington.

\begin{prop} The set $W$ is bijective to the set of non-trivial solutions of system (2.1).
\end{prop}

\begin{proof}
Send any weight homomorphism $w$ to the $n$-tuple $$(w(e_1),w(e_2),...,w(e_n)).$$ This mapping is obviously a bijection.
\end{proof}

Proposition 2.1 immediately implies the following observation by Etherington.

\begin{prop} (Etherington \cite{E})
An algebra $A$ is baric if and only if the system of equations (2.1) has a non-trivial solution.
\end{prop}

We call system (2.1) of equations the \textit{Etherington's system} of an algebra $A$ with respect to the basis $e_1, e_2,..., e_n$, and denote the set of its non-trivial solutions by $Eth(e_1,e_2,...,e_n)$.

\begin{cor}
For any bases $e_1,e_2,..., e_n$ and $e'_1,e'_2,..., e'_n$ of $A$, the sets $Eth(e_1,e_2,...,e_n)$ and $Eth(e'_1,e'_2,...,e'_n)$ are bijective.
\end{cor}

\begin{theo}(Schafer \cite{S})
An $n$-dimensional algebra $A$ is baric if and only if it has a basis $e_1,e_2,..., e_n$ such that, for its structural constants $\gamma_{ijk}$, the equalities
\begin{equation}
\sum_{k=1}^{n}\gamma_{ijk}=1 \; \; \; (1\leq i,j\leq n)
\end{equation}
hold ($1$ is the unit of the field $K$).
\end{theo}

A basis of $A$ that satisfies equalities (2.2) is called \textit{semi-natural} \cite{Z}.

\begin{lem}
A basis $e_1,e_2,..., e_n$ of an algebra $A$ is semi-natural if and only if $$(1,1,...,1)\in Eth(e_1,e_2,..., e_n).$$
\end{lem}
\vskip+3mm

 Proposition 2.1 immediately implies
\begin{theo}
A finite-dimensional algebra has a unique weight homomorphism if and only if the Etherington's system has a unique non-trivial solution.
\end{theo}

Theorem 2.6 enables us to construct an example of a baric algebra that has a unique weight homomorphism, but does not satisfy Holgate's sufficient condition mentioned in the Introduction.

\begin{exmp}
Let $A$ be a 2-dimensional algebra over the field $\mathbb{R}$ of real numbers with a basis $e_1, e_2$ and  structural constants $\gamma_{ijk}$ ($1\leq i,j,k \leq 2$). Let $$e_1^2=e_1, \;  e_2^2=e_2,$$
and the following inequalities be satisfied:
$$\gamma_{121}\neq \gamma_{211}, \; \gamma_{122}\neq \gamma_{212};$$
$$\gamma_{121}+\gamma_{122}\neq \gamma_{211}+\gamma_{212}.$$
These inequalities ensure that the corresponding Etherington's system has only the zero solution.

Consider the Cartesian product $A\times \mathbb{R}$ and its standard basis  \begin{equation}
\label{eqn:(2.4)}
(e_1,0), (e_2,0),(0, 1).
\end{equation} 
\noindent One obviously has \begin{equation}
\label{eqn:(2.8)}
(e_i,0)\cdot (e_j,0)= \gamma_{ij1}(e_1,0)+\gamma_{ij2}(e_2,0) \; \; (1\leq i,j\leq 2)
\end{equation}
This implies that all equations (with the unknowns $x_1,x_2$) from the Etherington's system (of the algebra $A$) with respect to the basis $e_1, e_2$ are involved in the algebra $A\times \mathbb{R}$'s Etherington's system (with the unknowns $x_1,x_2, x_{3}$) with respect to the basis (2.3). Applying this fact, it is easy to notice that the algebra $A\times \mathbb{R}$'s Etherington's system has precisely one non-trivial solution $(0,0,1)$.
Theorem 2.6 implies that the projection $\pi_2:A\times \mathbb{R}\rightarrow \mathbb{R}$ is a unique weight homomorphism of $A\times \mathbb{R}$. Its kernel obviously is not nil. 
\end{exmp}

\vskip+2mm
The following lemma is obvious.

\begin{lem}
Let  $e_1,e_2,...,e_n$ be a basis of an algebra $A$. For any element $\alpha\neq 0$ of $K$, the following conditions are equivalent:
\vskip+2mm
(i) the $n$-tuple $(\alpha,\alpha,...,\alpha)$ is a solution of the Etherington's system;

(ii) for any $i,j$ ($1\leq i,j,\leq n$), one has the equality

$$\sum_{k=1}^{n}\gamma_{ijk}=\alpha.$$
\end{lem}

Lemma 2.8 implies  
\begin{lem}
Let an algebra $A$ has a basis $e_1,e_2,...,e_n$. If, for any solution $(\alpha_1,\alpha_2,...,\alpha_n)$ of the Etherington's system, one has 
$$\alpha_1=\alpha_2=...=\alpha_n,$$
then $A$ has not more than one weight homomorphism.
\end{lem}

\begin{cor}
Let a baric algebra $A$ has a basis $e_1,e_2,...,e_n$ such that its structural constants $\gamma_{ijk}$ satisfy the equality
$$\gamma_{ijk}=\gamma_{ij'k},$$
for any $i,j,j',k$ ($1\leq i,j,j',k\leq n$). Then $A$ has a unique weight homomorphism.
\end{cor}

\begin{proof}
Let $(\alpha_1,\alpha_2,...,\alpha_n)$ be a solution of the Etherington's system, and let $\alpha_i\neq 0$. The equality $\alpha_{i}^{2}=\alpha_i\alpha_j$ implies that $\alpha_i=\alpha_j$, for any $j$. Now it suffices to apply Lemma 2.9.\end{proof}

\begin{rem}
Note that, in the case where $e_1,e_2,...,e_n$ is semi-natural and all structural coefficients are equal to one another (and are equal to, say, $\gamma$),  the claim of Corollary 2.10 follows also from the sufficient condition for the uniqueness of the weight homomorphism by Holgate mentioned in the Introduction. Indeed, let $\omega$ be a homomorphism with $\omega(e_i)=1$, for all $i$; it is an algebra homomorphism. Let $a\in Ker \; \omega$ and $a=\sum_{i=1}^{n}\alpha_i e_i$. Then $\omega(a)=\sum_{i=1}^{n}\alpha_i=0,$ and we have
$$a^2=(\sum_{i=1}^{n}\alpha_i e_i)^{2}=\sum_{k=1}^{n}\sum_{m=1}^{n}(\alpha_k\alpha_m)(e_ke_m)=$$
$$\gamma(\sum_{k=1}^{n}\sum_{m=1}^{n}\alpha_k\alpha_m)\sum_{i=1}^{n}e_i=\gamma(\sum_{i=1}^{n}\alpha_i)^{2}\sum_{j=1}^{n}e_j=0.$$
Therefore, $Ker \ \omega$ is nil.
\end{rem}

Finally, with the aid of Theorem 2.6, we provide an example of a baric algebra with two weight homomorphisms. Note that it does not fall under the scope of the Etherington's sufficient condition mentioned in the Introduction as the algebra considered in this example is not commutative.

\begin{exmp}

Let $A$ be the 3-dimensional algebra with the following Etherington's system in some basis:
$$x_{1}^{2}=x_2,\;  x_2^2=x_2, \; x_3^2= x_2,$$
$$x_1x_2= x_1, x_2x_1=x_3, x_1x_3=x_2, x_3x_1=x_2, x_2x_3=x_3, x_3x_2=x_3.$$

It is easy to see that this system has precisely two non-trivial solutions $(1,1,1)$ and $(-1,1,-1)$, and hence the algebra $A$ has two weight homomorphisms.  
\end{exmp}

\begin{rem}
In view of Example 2.12, observe that a $2$-dimensional non-commutative baric algebra has precisely one weight homomorphism. Indeed, the Etherington's system yields the equation
$$\gamma_{121}x_1+\gamma_{122}x_2=\gamma_{211}x_1+\gamma_{212}x_2.$$
Since $A$ is not commutative, either $\gamma_{121}\neq \gamma_{211}$ or $\gamma_{122}\neq\gamma_{212}$. Assume that the former inequality holds. Then $x
_1$ can be expressed as a constant multiplied by $x_2$. Assuming that $x_2\neq0$ and substituting this expression for $x_1$ to the equation
$$x_2^2=\gamma_{221}x_1+\gamma_{222}x_2,$$
 we obtain a unique value of $x_2$.
 
\end{rem}

\section{Weight homomorphisms, semi-natural bases, and transition matrices}
 Let $A$ be an $n$-dimensional algebra, and let $e_1,e_2,..., e_n$ and $f_1,f_2,...,f_n$ be its bases. We use the term `transition matrix from the basis $e_1,e_2,..., e_n$ to the basis $f_1,f_2,...,f_n$' for the matrix $(\sigma_{ik})_{1\leq i,k\leq n}$, where
\begin{equation}
f_i=\sum_{k=1}^{n}\sigma_{ik}e_k,
\end{equation}
\noindent for any $i$ ($i=1,...,n)$. 

\begin{prop}
Let $A$ be an $n$-dimensional algebra, and let $e_1,e_2,..., e_n$ and $f_1,f_2,...,f_n$ be its bases. Let $M=(\sigma_{ij})_{1\leq i,j\leq n}$ be the transition matrix from $e_1,e_2,..., e_n$ to  $f_1,f_2,...,f_n$. The basis $e_1,e_2,..., e_n$ is semi-natural if and only if the $n$-tuple $(\alpha_1,\alpha_2,...,\alpha_n)$ with 
\begin{equation}
\alpha_i=\sum_{k=1}^{n}\sigma_{ik} \; \; \; (1\leq i\leq n)
\end{equation} 
 is a solution of the Etherington's system of equations with respect to the  basis $f_1,f_2,..., f_n$.

\end{prop}

\begin{proof}"Only if": Let the structural constants of the bases $e_1,e_2,..., e_n$ and $f_1,f_2,...,f_n$ be resp. $\gamma_{ijk}$ and $\xi_{ijk}$ $(1\leq i,j,k\leq n)$. 
We have 
\begin{equation}
f_if_j=(\sum_{k=1}^{n}\sigma_{ik}e_k)(\sum_{l=1}^{n}\sigma_{jl}e_l)=\sum_{m=1}^{n}\sum_{k=1}^{n}\sum_{l=1}^{n}(\sigma_{ik}\sigma_{jl}\gamma_{klm})e_m.
\end{equation}
Further, considering the fact that the basis $e_1,e_2,...,e_n$ is semi-natural, we obtain
$$\sum_{m=1}^{n}\sum_{k=1}^{n}\sum_{l=1}^{n}(\sigma_{ik}\sigma_{jl}\gamma_{klm})=\sum_{k=1}^{n}\sum_{l=1}^{n}\sigma_{ik}\sigma_{jl}(\sum_{m=1}^{n}\gamma_{klm})=\sum_{k=1}^{n}\sum_{l=1}^{n}\sigma_{ik}\sigma_{jl}=$$
\begin{equation}
\sum_{k=1}^{n}\sigma_{ik}\sum_{l=1}^{n}\sigma_{jl}=\alpha_i\alpha_j.
\end{equation}
On the other hand, we have
\begin{equation}
f_if_j=\sum_{m=1}^{n}\xi_{ijm}f_{m}=\sum_{m=1}^{n}\xi_{ijm}\sum_{k=1}^{n}\sigma_{mk}e_k=\sum_{k=1}^{n}(\sum_{m=1}^{n}\xi_{ijm}\sigma_{mk})e_k.
\end{equation}
The sum of coefficients in expression (3.5) is
\begin{equation}
\sum_{k=1}^{n}\sum_{m=1}^{n}\xi_{ijm}\sigma_{mk}=\sum_{m=1}^{n}\xi_{ijm}\sum_{k=1}^{n}\sigma_{mk}=\sum_{m=1}^{n}\xi_{ijm}\alpha_m.
\end{equation}
Equalities (3.3)-(3.6) imply that the $n$-tuple $(\alpha_1,\alpha_2,...\alpha_n)$ is a solution of the Etherington's system of equations with respect to the basis $f_1,f_2,...,f_m$.

"If": Let the inverse of the matrix $M$ be $M^{-1}=(\varrho_{ij})_{1\leq i,j\leq n}$. Then
\begin{equation}
e_i=\sum_{k=1}^{n}\varrho_{ik}f_k.
\end{equation}
Hence
\begin{equation}
e_ie_j=\sum_{m=1}^{n}\sum_{k=1}^{n}\sum_{l=1}^{n}(\varrho_{ik}\varrho_{jl}\xi_{klm})f_m=\sum_{m=1}^{n}\sum_{k=1}^{n}\sum_{l=1}^{n}(\varrho_{ik}\varrho_{jl}\xi_{klm})\sum_{p=1}^{n}\sigma_{mp}e_p=
\end{equation}
$$\sum_{p=1}^{n}\sum_{k=1}^{n}\sum_{l=1}^{n}(\varrho_{ik}\varrho_{jl}\sum_{m=1}^{n}\xi_{klm}\sigma_{mp})e_p.$$
Since $(\alpha_1,\alpha_2,...,\alpha_n)$ is a solution of the Etherington's system, we have
$$\sum_{p=1}^{n}\sum_{k=1}^{n}\sum_{l=1}^{n}\varrho_{ik}\varrho_{jl}\sum_{m=1}^{n}\xi_{klm}\sigma_{mp}=\sum_{k=1}^{n}\sum_{l=1}^{n}\varrho_{ik}\varrho_{jl}\sum_{m=1}^{n}\xi_{klm}\sum_{p=1}^{n}\sigma_{mp}=$$
$$\sum_{k=1}^{n}\sum_{l=1}^{n}\varrho_{ik}\varrho_{jl}\sum_{m=1}^{n}\xi_{klm}\alpha_m=\sum_{k=1}^{n}\sum_{l=1}^{n}\varrho_{ik}\varrho_{jl}\alpha_k\alpha_l=$$
$$\sum_{k=1}^{n}\varrho_{ik}\alpha_k\sum_{l=1}^{n}\varrho_{jl}\alpha_l.$$
Recall that the matrices $(\sigma_{ij})_{1\leq i,j\leq n}$ and $(\varrho_{ij})_{1\leq i,j\leq n}$ are inverses of each other. This implies that $$\sum_{k=1}^{n}\varrho_{ik}\alpha_k=1.$$ Thus the sum of coefficients in (3.8) is equal to 1.
\end{proof}

\begin{lem}
Let $\alpha_1,\alpha_2,...,\alpha_n$ be elements of a field $K$. There is a non-singular matrix $M=(\sigma_{ik})_{1\leq i,k\leq n}$ such that
\begin{equation}
\sum_{k=1}^{n}\sigma_{ik}=\alpha_i,
\end{equation}
\noindent for any $i$, if and only if at least one $\alpha_i$ is not zero. 
\end{lem}

\begin{proof} 
"Only if": If all of $\alpha_i$ were equal to zero, then the columns of $M$ would be lineraly dependent, which would be a contradiction.

"If": We apply the principle of the mathematical induction by $n$. For $n=1, 2$, the statement is obvious. Assume that $n\geq 3$, and that the statement is valid for $n-1$. Let $a_1\neq 0$, and $M'$ be a non-singular $
(n-1)\times (n-1)$ matrix such that the sum of entries in its rows are equal to resp. $\alpha_2, \alpha_3,...,\alpha_n$. Let $M$ be the $n\times n$ matrix obtained from $M'$ by adding an $n\times 1$ column on the left and an $1\times n$ row above $M'$ such that the first entry of the added column and row is $\alpha_1$, while all other entries are zeros.


If $a_1=0$ and $a_k\neq 0$, then we can construct the sought-for-matrix $M$ for the sequence $\alpha_k,\alpha_2,...,\alpha_1,...\alpha_n$ (where $a_1$ presents on the $k$th position), and then permute the first and the $k$th rows in $M$.

\end{proof}

Recall that a matrix is called row stochastic (resp. column  stochastic) if the sum of entries in any row (resp. column) is equal to 1. The following lemma is well-known (see e.g., \cite{P}).

\begin{lem} 
The set $RS_n(K)$ (resp. $CS_n(K)$) of non-singular $n\times n$ row stochastic (resp. column stochastic) matrices is a subgroup of the general linear group $GL_n(K)$. 
\end{lem}

\begin{theo}
Let $A$ be a baric algebra, and $e_1,e_2,...,e_n$ be its semi-natural basis. The following conditions are equivalent:

(i) $A$ has a unique weight homomorphism;

(ii) the transition matrix from any semi-natural basis to $e_1,e_2,...,e_n$ is row stochastic;
\vskip+1mm
(iii)  the transition matrix from any semi-natural basis $e'_1,e'_2,...,e'_n$ to any semi-natural basis $e''_1,e''_2,...,e''_n$ is row stochastic.
\end{theo}

\begin{proof}
The implications (i)$\Rightarrow$(ii) and (i)$\Rightarrow$(iii) follow from Theorem 2.6 and Proposition 3.1. 

(ii)$\Rightarrow$(i): Let $$(\alpha_1,\alpha_2,...,\alpha_n)\in Eth(e_1,e_2,...,e_n).$$ By Lemma 3.2, there is a non-singular matrix $M=(\sigma_{ik})_{1\leq i,k\leq n}$ such that equality (3.9) holds for any $i$ ($i=1,2,...,n)$. Let $e'_1,e'_2,...,e'_n$ be the basis of $A$ such that the transition matrix from it to the basis $e_1,e_2,...,e_n$ is $M$. According to Proposition 3.1, the basis $e'_1,e'_2,...,e'_n$ is semi-natural. Therefore, the matrix $M$ is row stochastic, and hence $\alpha_i=1$, for any $i$ ($i=1,2,...,n)$. Now it suffices to apply Theorem 2.6.

The implication (iii)$\Rightarrow$(ii) is obvious.
\end{proof}

\
\section{Additional remarks}
Let $$\mathbf{F}:GL_n(K)\rightarrow K^n\setminus \lbrace (0,0,...,0)\rbrace$$
\noindent  be the mapping sending a non-singular matrix $M=(\sigma_{ij})_{1\leq i,j\leq n}$ to the $n$-tuple $(\alpha_1,\alpha_2,...,\alpha_n)$ with
\begin{equation}
\alpha_i=\sum_{k=1}^{n} \sigma_{ik}.
\end{equation}
\noindent Lemma 3.2 ensures that all $\alpha_i$'s cannot be equal to zero. Note that if $n>1$, then the mapping $\mathbf{F}$ is not a homomorphism (of monoids) (for instance, $\mathbf{F}$ does not preserve the product of the matrix where all entries of the first column are $1$, while all other entries are $0$, and an arbitrary matrix $M$ with $\alpha_1\neq \alpha_2$).

Lemma 3.2 immediately implies

\begin{lem}
The mapping $\mathbf{F}$ is surjective.
\end{lem}

Let $A$ be an $n$-dimensional algebra, and  $S$ be the set of semi-natural bases of $A$. Let $f_1,f_2,...,f_n$ be a basis of $A$, and
$$\mathbf{M}: S\rightarrow GL_n(K)$$
be the mapping that sends a semi-natural basis $e_1,e_2,...,e_n$ to the transition matrix $M=(\sigma_{ik})_{1\leq i,k\leq n}$ from this basis to $f_1,f_2,...,f_m$.
\vskip+1mm
Proposition 3.1 and Lemma 3.2 imply 

\begin{lem} There is a unique mapping $\mathbf{F}'$ rendering the following square a pullback (of sets):
\begin{equation}
\xymatrix{S \; \ar@{>->}[r]^{\mathbf{M}}\ar[d]_{\mathbf{F
}'}&GL_n(K)\ar[d]^{\mathbf{F}}\\
Eth(f_1,f_2,...,f_m) \; \ar@{>->}[r]^-{\mathbf{I}}&K^n\setminus \lbrace (0,0,...,0)\rbrace}
\end{equation}
Here $\mathbf{I}$ is the embedding mapping.

\end{lem}

\begin{cor}
For any basis $f_1,f_2,...,f_n$ of $A$, the set $S$ is bijective to the inverse image of $Eth(f_1,f_2,...,f_m)$ under $\mathbf{F}$. 
\end{cor}

Consider the mapping $$\mathbf{G}:S\rightarrow W$$ that sends a semi-natural basis $e_1, e_2,...,e_n$ to the weight homomorphism $w$ defined by $w(e_i)=1$ ($i=1,2,...,n)$; $\omega$ is indeed an algebra homomorphism, as it is easy to notice.

\begin{lem}
Let $e_1, e_2,...,e_n$ and $e'_1, e'_2,...,e'_n$  be semi-natural bases of an algebra $A$, and let the transform matrix from $e_1, e_2,...,e_n$ to $e'_1, e'_2,...,e'_n$ be $N$. Then the following conditions are equivalent:
\vskip+1mm
(i) $\mathbf{G}(e_1, e_2,...,e_n)=\mathbf{G}(e'_1, e'_2,...,e'_n)$;
\vskip+1mm
(ii) there is a weight homomorphism $w$ such that the matrix of $w$ in both bases is $(1,1,...,1)$;
\vskip+1mm
(iii) the following  equality holds:
\begin{equation}
N(1,1,...,1)^T=(1,1,...,1)^{T}.
\end{equation}
\vskip+1mm
(iv) $N$ is row stochastic.
\vskip+1mm
\end{lem}

\vskip+3mm
Recall that, according to Proposition 2.1, for any basis $f_1,f_2,...,f_n$ of $A$, there is a bijection $$\mathbf{H}:W\rightarrow Eth(f_1,f_2,...,f_n).$$
\begin{lem}
For any basis $f_1,f_2,...,f_n$ of $A$, one has the equality $$\mathbf{F}'=\mathbf{H}\mathbf{G}.$$
\end{lem}

\begin{proof}
It suffices to note that $$\mathbf{F}'(e_1, e_2,...,e_n)=$$ $$(\mathbf{G}(e_1, e_2,...,e_n)(f_1), \mathbf{G}(e_1, e_2,...,e_n)(f_2),...,\mathbf{G}(e_1, e_2,...,e_n)(f_n)).$$
 \end{proof} 
 
 \begin{lem}
 The mapping $\mathbf{G}$ is surjective.
 \end{lem}

\begin{proof}
The claim immediately follows from Lemmas 4.1, 4.2 and 4.5. It follows also from the proof of Lemma 1.11($(1)\Rightarrow (2)$) of \cite{W}. 
\end{proof}

Lemmas 4.4 and 4.6 imply the following statement.
\begin{prop}
The set $W$ of weight homomorphisms of $A$ is bijective to the quotient of the set $S$ by the equivalence relation induced by the left coset division of the general linear group $GL_n(K)$ by the subgroup $RS_n(K)$ of row stochastic matrices. 
\end{prop}

\begin{proof}
Let $f_1,f_2,...,f_n$ be a basis of $A$, while  $e_1, e_2,...,e_n$ and $e'_1, e'_2,...,e'_n$  be semi-natural bases of $A$. Let the transition matrix from $e_1, e_2,...,e_n$ to $e'_1, e'_2,...,e'_n$ be $N$, while the transition matrices from $e_1, e_2,...,e_n$ and $e'_1, e'_2,...,e'_n$ to $f_1,f_2,...,f_n$ be $M$ and $M'$ respectively. Now the claim follows from the obvious equality $M=M'N$.
\end{proof}

\begin{cor}
Let $f_1,f_2,...,f_m$ be a basis of an algebra $A$. The conditions (i) and (ii) below are equivalent and are implied by the equivalent conditions (iii), (iv). If $f_1,f_2,...,f_m$ is semi-natural, then all four conditions are equivalent:

(i) $A$ has a unique weight homomorphism;

(ii) for any semi-natural bases, their transition matrices to the basis $f_1,f_2,...,f_n$ lie in the same left coset with respect to the subgroup $RS_n(K)$ of $GL_n(K)$;

(iii) the image of $S$ under the mapping $\mathbf{M}$ is $RS_n(K)$;

(iv) one has $$\mathbf{F}^{-1}(Eth(f_1,f_2,...,f_n))=RS_n(K).$$

\begin{proof}
The equivalence (i)$\Leftrightarrow$(ii) immediately follows from Proposition 4.7. The equivalence (iii)$\Leftrightarrow$(iv) follows from Lemma 4.2. The implication (iii)$\Rightarrow$(ii) is obvious. If $f_1,f_2,...,f_m$ is semi-natural, then the equivalence (i)$\Leftrightarrow$(iv) follows from Lemma 2.5 and Theorem 2.6.
\end{proof}

\end{cor}

\vskip+3mm

\textit{Author's address:}

\textit{Dali Zangurashvili, A. Razmadze Mathematical Institute of Tbilisi State University,}
\textit{2 M. Aleksidze Str., Lane II, Tbilisi 0193, Georgia}

\textit{e-mail:} dali.zangurashvili@tsu.ge


\begin{thebibliography}{99}


\bibitem{E} I. M. H. Etherington,  Genetic algebras. Proc. Roy. Soc. Edinburgh 59 (1939) 242-258.

\bibitem{H} P. Holgate, Characterizations of genetic algebras. J.
London Math. Soc. 2(1972) 169-174.


\bibitem{P} D. G. Poole, The Stochastic Group, Amer. Math.
Monthly 798-801 (1995).


\bibitem{S} R. D. Schafer, Structure of genetic algebras. Amer. J.
Math. 71 (1949) 121-135.

\bibitem{W} A. W\"orz-Busekros, Algebras in genetics, Lecture Notes in Biomath. 36, Springer-Verlag  1980.

\bibitem{Z} D. Zangurashvili, The binary products of algebras with genetic realization, Georgian Math. Journal, https://doi.org/10.1515/gmj-2024-2060, 2024.


\end{thebibliography}
\end{document}